\documentclass{amsart} 
\usepackage{amssymb,latexsym,textcomp}
\usepackage{mathrsfs}

\usepackage[OT2,T1]{fontenc}

\theoremstyle{plain}
\newtheorem{theorem}{Theorem}[section]
\newtheorem{lemma}[theorem]{Lemma}
\newtheorem{corollary}[theorem]{Corollary}

\theoremstyle{definition}

\theoremstyle{remark}
\newtheorem{remark}{Remark}[section]

\DeclareMathOperator{\tr}{tr}%

\newcommand{\Z}{\mathbf Z}
\newcommand{\R}{\mathbf R}
\newcommand{\T}{\mathbf T}

\newcommand{\e}{\varepsilon}
\renewcommand{\d}{\mathrm d}
\newcommand{\skp}[2]{\left<#1,#2\right>}
\newcommand{\C}{\mathbf C}
\newcommand{\Q}{\mathbb Q}
\renewcommand{\a}{\alpha}

\newcommand{\ovl}{\overline}
\newcommand{\Sc}{\mathcal C}
\newcommand{\Gl}{\mathrm{Gl}}
\newcommand{\re}{\mathrm{Re}}

\numberwithin{equation}{section} 

\begin{document}
\title[Clarkson-McCarthy inequality on a locally compact group]{Clarkson-McCarthy inequality on a locally compact group} 

\author{Dragoljub J. Ke\v cki\' c, Zlatko Lazovi\' c}
\address{University of Belgrade\\ Faculty of Mathematics\\ Student\/ski trg 16-18\\ 11000 Beograd\\ Serbia}

\email{keckic@matf.bg.ac.rs, zlatko.lazovic@matf.bg.ac.rs}

\thanks{This work is partially supported by the Ministry of Science, Technological Development and Innovation of Republic of Serbia: grant number 451-03-47/2023-1/200104 with Faculty of Mathematics.}

\begin{abstract}
Let $G$ be a locally compact group, $\mu$ its Haar measure, $\hat G$ its Pontryagin dual and $\nu$ the dual measure. For any $A_\theta\in L^1(G;\Sc_p)\cap L^2(G;\Sc_p)$, ($\Sc_p$ is Schatten ideal), and $1<p\le2$ we
prove
$$\int_{\hat G}\left\|\int_GA_\theta\ovl{\xi(\theta)}\d\mu(\theta)\right\|_p^q\d\nu(\xi)\le
    \left(\int_G\|A_\theta\|_p^p\d\mu(\theta)\right)^{q/p},
$$
where $q=p/(p-1)$. This appears to be a generalization of some earlier obtained inequalities, including Clarkson-McCarthy inequalities (in the case $G=\Z_2$), and Hausdorff-Young inequality. Some corollaries are also given.
\end{abstract}


\subjclass{47A30, 47B10, 43A25}

\keywords{Clarkson inequalities, Hausdorff-Young inequality, unitarily invariant norm, abstract Fourier transform}

\maketitle

\section{Introduction}

Investigating uniformly convex spaces, Clarkson \cite{Clarkson} proved the following inequalities for $L^p$ norms:
\begin{equation}\label{Lp>2}
\begin{gathered}
2(\|f\|_p^p+\|g\|_p^p)\le\|f+g\|_p^p+\|f-g\|_p^p\le2^{p-1}(\|f\|_p^p+\|g\|_p^p),\qquad p\ge2\\
2(\|f\|_p^p+\|g\|_p^p)\ge\|f+g\|_p^p+\|f-g\|_p^p\ge2^{p-1}(\|f\|_p^p+\|g\|_p^p),\qquad p\le2
\end{gathered}
\end{equation}
\begin{equation}\label{Lp<2alt}
\|f+g\|_p^p+\|f-g\|_p^p\le2(\|f\|_p^q+\|g\|_p^q)^{p/q},\qquad p\ge 2,\; q=p/(p-1).
\end{equation}
\begin{equation}\label{Lp<2}
\|f+g\|_p^q+\|f-g\|_p^q\le2^{q-1}(\|f\|_p^p+\|g\|_p^p)^{q/p},\qquad p\le 2,\; q=p/(p-1).
\end{equation}

Later McCarthy \cite{McCarthy} generalized these inequalities to Schatten classes of operators. He replaced measurable functions $f$ and $g$ by compact operators $A$ and $B$, and $L^p$ norm by $C_p$ norm defined as
$$\|A\|_p=\left(\tr(|A|^p)\right)^{1/p}.$$
The inequalities he obtained was exactly the inequalities (\ref{Lp>2}), (\ref{Lp<2alt}) and (\ref{Lp<2}). In operator framework, these inequalities are usually referred as Clarkson-McCarthy inequalities.

Bhatia and Kittaneh \cite{BhatiaFudo}, proved the following generalization for $n$-tuple of operators. Among others, for $1<p\le 2$, they proved in Theorem 1.4:
\begin{equation}\label{Cpn}
\sum_{k=0}^{n-1}\Big\|\sum_{j=0}^{n-1}\omega_j^kA_j\Big\|_p^q\le n\left(\sum_{j=0}^{n-1}\|A_j\|_p^p\right)^{\frac qp},
\end{equation}
where $\omega_j=e^{2\pi i j/n}$ is the $j$-th degree of the $n$-th root of unity.

The first author in \cite{Keckic2019} generalized these inequalities to compact topological groups. Among others, he proved in Theorem 3.6:
\begin{equation}\label{Inequalityp<2Compact}
\sum_{k\in\hat G}\left\|\int_G\ovl{k(\theta)}A_\theta\d\mu(\theta)\right\|_p^q\le
    \left(\int_G\|A_\theta\|_p^p\right)^{q/p},
\end{equation}
for compact abelian group $G$ and the corresponding Haar measure $\mu$. This inequality is a generalization of both(\ref{Lp<2}) if $G=\Z_2$ and (\ref{Cpn}) if $G=\Z_n$.

In the same paper in Problem 5.4, it is asked whether it is possible to generalize (\ref{Inequalityp<2Compact}) to locally compact groups. It is also explained why other Clarkson-McCarthy inequalities (e.g.\ (\ref{Lp>2}) and (\ref{Lp<2alt})) can't be generalized to locally compact groups in a straightforward way.

The aim of this paper is to generalize (\ref{Inequalityp<2Compact}) to locally compact groups. The main result is contained in Theorem \ref{TheoremPQ}.

The paper is organized as follows: in section 2, we give a brief recapitulation of topics, necessary for further reading. Section 3 is devoted to main results. Finally, in section 4, we derive applications of the main results.

\section{Prerequisits}

\subsection{Schatten ideals} Let $H$ be a separable Hilbert space, and let $A:H\to H$ be a compact operator. Its singular values $s_n(A)$ are defined as the eigenvalues of its absolute value $(A^*A)^{1/2}$, i.e.~$s_n(A)=\lambda_n(A^*A)^{1/2}$.

The Schatten ideala $\Sc_p$ ($1\le p<\infty$) are defined as
$$\Sc_p=\{A:H\to H\mid A\mbox{ compact},\,\|A\|_p=\tr(|A|^p)^{1/p}=\left(\sum_{n=1}^\infty s_n^p(A)\right)^{1/p}<\infty\}.$$
For $p=\infty$, $\Sc_\infty$ is the set of all compact operators, with the usual norm inherited from $B(H)$.

It is well known that all $\Sc_p$ are Banach spaces, ideals in $B(H)$, aa well as the following duality results
$$(\Sc_p)^*\cong\Sc_q,\quad 1<p<\infty\quad\frac1p+\frac1q=1,\qquad(\Sc_1)^*\cong B(H),\qquad (\Sc_\infty)^*\cong B(H).$$

For more details, see \cite[Chapter III]{GKrein}.

\subsection{Abstract harmonic analysis} Let $(G,+)$ be an abelian group with locally compact Hausdorff topology such that the operations in $G$ are continuous with respect to it. As usual, we shall use the name \textit{locally compact group}.

\textit{Haar measure} on a locally compact group is a regular, Borel and invariant measure $\mu$ on $G$. (The latter means $\mu(E)=\mu(x+E)$.) Haar measure is unique up to a positive multiplicative scalar. If $G$ is compact it is usual to normalize the Haar measure by $\mu(G)=1$. However, if $G$ is only locally compact then, in general, there is no natural normalization.

A \textit{character} on $G$ is a homomorphism from $(G,+)$ to $(\T,\cdot)$ where $\T=\{z\in\C\mid|z|=1\}$, i.e.\ a mapping $\xi:G\to\T$ such that $\xi(x+y)=\xi(x)\xi(y)$.

\textit{Fourier transform} or \textit{abstract Fourier transform} of $f\in L^1(G)$ is
$$\hat f(\xi)=\int_G f(t)\overline{\xi(t)}\d\mu(t),$$
where $\xi\in\hat G$ and $\mu$ is the Haar measure on $G$.

The \textit{Pontryagin dual} of $G$ denoted by $\hat G$ is the set of all characters. It becomes a locally compact group if it is endowed by \textit{compact-open} topology. Therefore, $\hat G$ has, also, the Haar measure, denote by $\nu$ or $\hat\mu$. It is normalized such that the following \textit{inversion formula} holds:
$$f(t)=\int_{\hat G}\hat f(\xi)\xi(t)\d\nu(\xi),\qquad\mbox{for }f\in L^1(G)\mbox{ and }\hat f\in L^1(\hat G).$$

Many statements, such as Plancharel theorem, can be generalized to the abstract harmonic analysis. For more details see \cite{Foland}.

\subsection{Bochner spaces}

Let $(\Omega,\mu)$ be a measurable space and let $X$ be a Banach space. The \textit{Bochner space} $L^p(\Omega;X)$ for $1\leq p<\infty$ is defined as the set of strongly measurable functions $f:\Omega\to X$ such that
$$\|f\|_{L^p(\Omega;X)}:=\left(\int_\Omega\|f(t)\|_X^p\d\mu(t)\right)^{1/p}<+\infty,$$
after identification of $\mu$-almost everywhere equal functions. Here, strong measurability is equivalent to weak measurability (that is, the measurability of scalar functions $t\mapsto\Lambda(f(t))$ for all $\Lambda\in X^*$) and separability of the image of $f$.

In a similar way $L^{\infty}(\Omega;X)$ can be defined, with
$$\|f\|_{L^{\infty}(\Omega;X)}= \inf \{r\mid\mu\{\|f\|_X>r\}=0\}.$$

Bochner integral is linear, additive with respect to disjoint union, and also there holds
\begin{equation}\label{BochnerLinear}
T\int_\Omega f(t)\d\mu(t)=\int_\Omega Tf(t)\d\mu(t)
\end{equation}
for all $f\in L^p(\Omega;X)$ and all bounded linear $T:X\to Y$, provided that the left hand side exists.

In particular, for a Hilbert space $H$, the mapping $B(H)\ni A\mapsto\skp{Ag}h\in\C$, for fixed $g$, $h\in H$ is linear and bounded, and from (\ref{BochnerLinear}) we obtain
\begin{equation}\label{BochnerInner}
\skp{\left(\int_\Omega A_t\d\mu(t)\right)g}{h}=\int_\Omega\skp{A_tg}{h}\d\mu,
\end{equation}
for any $A_t\in L^1(\Omega;B(H))$ and all $g$, $h\in H$.

\begin{remark}
    Note a slightly more general result in \cite[Lemma 1.2]{Krtinic} that assumes only weak-* measurability of $A_t$.
\end{remark}

For more details, the reader is referred to \cite[Chapters 1 and 2]{Banah}.

\subsection{Complex interpolation}

Let us briefly explain basic concepts of complex interpolation. For more detailed approach the reader is referred to \cite{Bergh} (see also \cite[Appendix C]{Banah}).

Given two Banach spaces, say $X$ and $Y$, both of them continuously embedded in some Hausdorff topological vector space $V$, consider the spaces $X\cap Y$ and $X+Y$ with norms given by
$$\|a\|_{X\cap Y}=\max\{\|a\|_X,\|a\|_Y\},\qquad \|a\|_{X+Y}=\inf_{\substack{a=a_1+a_2\\a_1\in X,a_2\in Y}}(\|a_1\|_X+\|a_2\|_Y).$$
Such a pair $(X,Y)$ is called \textit{a compatible pair}.

Consider the set $\mathcal F$ consisting of all functions $f:\overline S\to X+Y$, analytic and bounded in the strip $S=\{z\in\C\mid 0<\re z<1\}$ and continuous on the closure $\overline S$ such that:
$$f(it)\in X,\quad f(1+it)\in Y,\qquad\mbox{for all }t\in\R.$$
The set $\mathcal F$ is, obviously, a linear space and moreover a Banach space if the norm is given by
$$\|f\|=\max\{\sup_{t\in\R}\|f(it)\|_X,\sup_{t\in\R}\|f(1+it)\|_Y\}.$$

Now, the interpolation space $[X,Y]_\theta$, $\theta\in[0,1]$ is defined by
$$[X,Y]_\theta=\{a\in X+Y\mid\mbox{there is an }f\in\mathcal F,f(\theta)=a\},$$
and the norm that makes it a Banach space, is given by
$$\|a\|_\theta=\inf_{f\in\mathcal F,f(\theta)=a}\|f\|_{\mathcal F}.$$

The space $[X,Y]_\theta$ has the property of being \textit{exact interpolation space}. This is described in the following theorem.

\begin{theorem}\label{ExactInterpolation}
    Let $(X_1,Y_1)$ and $(X_2,Y_2)$ be two compatible pairs, and let $T:X_1+Y_1\to X_2+Y_2$ be a linear mapping, bounded $X_1$ to $X_2$, and from $Y_1$ to $Y_2$. Then $T$ is a bounded operator from $[X_1,Y_1]_\theta$ to $[X_2,Y_2]_\theta$ and the following estimate holds:
    $$\|T\|_{B([X_1,Y_1]_\theta,[X_2,Y_2]_\theta}\le\|T\|_{B(X_1,X_2)}^{1-\theta}\|T\|_{B(Y_1,Y_2)}^\theta.$$
\end{theorem}

\begin{proof}See \cite[Theorem 4.1.2]{Bergh} or \cite[Theorem C.2.6]{Banah}.\end{proof}

Besides this general theorem dealing with complex interpolation, we need two results concerning concrete examples of Banach space.

Firstly, it is well known that for classic Lebesgue spaces we have
$$(L^p(\Omega,\mu),L^q(\Omega,\mu))_\theta=L^r(\Omega,\mu),\qquad 1/r=(1-\theta)/p+\theta/q.$$
However, there is a more general result:

\begin{theorem} \label{T1}\cite[Theorem 2.2.6]{Banah}
    Let $1\leq p_1\leq p_2<\infty$ or  $1\leq p_1<=\infty$ and let $0<\theta<1$. For any interpolation couple $(X_1,X_2)$ of complex Banach spaces and any measure space $(\Omega,\mu)$ we have $$[L^{p_1}(\Omega;X_1),L^{p_2}(\Omega;X_2]_{\theta}=L^{p_{\theta}}(\Omega;[X_1,X_2]_{\theta})$$ isometrically, with $\frac{1}{p_{\theta}}=\frac{1-\theta}{p_1}+\frac{\theta}{p_2}$.
\end{theorem}

Also, we need the interpolation result for Schatten classes.

\begin{theorem}\cite[Theorem D.3.1]{Banah}
    Let $\mathcal C^p$, $1\le p\le +\infty$ denote the Schatten class of compact operators on some Hilbert space $H$. Then for all $\theta\in[0,1]$ we have
    $$[\mathcal C^{p_1},\mathcal C^{p_2}]_\theta=\mathcal C^p,$$
    where $1/p=(1-\theta)/p_1+\theta/p_2$.
\end{theorem}

As a corollary of preceding two theorems we have

\begin{corollary}\label{InterpolationLC}
    Let $1\le p_1$, $p_2$, $q_1$, $q_2\le+\infty$, let $\theta\in[0,1]$ and let $(\Omega,\mu)$ be some measure space. Then
    $$[L^{p_1}(\Omega;\mathcal C^{q_1}),L^{p_2}(\Omega;\mathcal C^{q_2})]_\theta=L^p(\Omega;\mathcal C^q),$$
    where $1/p=(1-\theta)/p_1+\theta/p_2$ and $1/q=(1-\theta)/q_1+\theta/q_2$.
\end{corollary}

\section{Main results}

First, we establish Parseval identity for functions from Bochner spaces, which is the key technical tool in this paper.

\begin{theorem}\label{ParsevalThm} Let $G$ be a locally compact abelian group, let $\hat G$ be its Pontryagin dual. For $A_\theta\in L^2(G;B(H))\cap L^1(G;B(H))$, and $\xi\in\hat G$, the operators
$$B_\xi=\int_G \overline{\xi(\theta)}A_{\theta}\d\mu(\theta)$$
are well defined and also
\begin{equation}\label{AbstractParseval}
\int_{\hat G}|B_\xi|^2\d\nu(\xi)=\int_G|A_\theta|^2\d\mu(\theta),
\end{equation}
\end{theorem}

\begin{proof}
Since $\||A_\theta|^2\|=\|A_\theta\|^2$ and $A_\theta\in L^2(G;B(H))$, the integral
$$\int_G|A_\theta|^2\d\mu(\theta)$$
exists. By (\ref{BochnerInner}) we obtain
$$\skp{\int_G|A_{\theta}|^2\d\mu(\theta)h}{h}=\int_G\skp{|A_{\theta}|^2h}{h}\d\mu(\theta).$$

On the other hand, since $|\xi(\theta)|=1$ and  $A_{\theta}\in L^1(G;B(H))$ we have
\begin{align*}
    \int_{G}\|\overline{\xi(\theta)}A_{\theta}\|\d\mu(\theta)=\int_{G}\|A_{\theta}\|\d\mu(\theta)<\infty
\end{align*}
and $B_\xi$ are well defined.

Let $\{e_n\}_{n=1}^{\infty}$ be an arbitrary orthonormal basis of $H$. Then for all $h\in H$
\begin{align*}
    \skp{\int_{G}|A_{\theta}|^2\d\mu(\theta)\,h}{h}&=\int_G\|A_{\theta }h\|^2\d\mu(\theta)=\int_{G}\sum_{m=1}^{\infty}\left|\skp{A_{\theta}h}{e_m}\right|^2\d\mu(\theta)\\
    &=\sum_{m=1}^{\infty}\int_{G}\left|\skp{A_{\theta}h}{e_m}\right|^2\d\mu(\theta).
\end{align*}
The function $f\colon G\to\C$, $f(\theta)=\skp{A_{\theta}h}{e_m}$ belongs to $L^1(G)\cap L^2(G)$ for every $h\in H$ and $e_m$ since
$$\int_G|\skp{A_{\theta}h}{e_m}|\d\mu(\theta)\leq\|h\|\int_G\|A_{\theta}\|\d\mu(\theta)<\infty$$
and
$$\int_G|\skp{A_{\theta}h}{e_m}|^2\d\mu(\theta)\leq\|h\|^2\int_G\|A_{\theta}\|^2\d\mu(\theta)<\infty.$$

By the Plancherel theorem (\cite[Theorem 4.26]{Foland}) we have
$$\int_G|f(\theta)|^2\d\mu(\theta)=\int_{\hat G}|\hat{f}(\xi)|^2\d\nu(\xi),$$
where $\hat{f}$ the Fourier transform of $f$ on $G$, and hence
$$\int_{G}\left|\skp{A_{\theta}h}{e_m}\right|^2\d\mu(\theta)=\int_{\hat G}\left|\int_G\overline{\xi(\theta)}\skp{A_{\theta}h}{e_m}\d\mu(\theta)\right|^2\d\nu(\xi).$$
Therefore,
\begin{align*}
    \sum_{m=1}^{\infty}\int_{G}\left|\skp{A_{\theta}h}{e_m}\right|^2\d\mu(\theta)&=\sum_{m=1}^{\infty}\int_{\hat G}\left|\int_G\overline{\xi(\theta)}\skp{A_{\theta}h}{e_m}\d\mu(\theta)\right|^2\d\nu(\xi)\\
&=\sum_{m=1}^{\infty}\int_{\hat G}\left|\int_G\skp{\overline{\xi(\theta)}A_{\theta}h}{e_m}\d\mu(\theta)\right|^2\d\nu(\xi)\\
&=\int_{\hat G}\sum_{m=1}^{\infty}\left|\skp{\int_G\overline{\xi(\theta)}A_{\theta}\d\mu(\theta)h}{e_m}\right|^2\d\nu(\xi).
\end{align*}
The last equality follows from (\ref{BochnerInner}) because $\skp{\overline{\xi(\theta)}A_{\theta}h}{h}\in L^1(G,\mu)$ for all $h\in H$. Hence,
\begin{align*}
    \skp{\int_{G}|A_{\theta}|^2\d\mu(\theta)\,h}{h}&=\int_{\hat G}\sum_{m=1}^{\infty}\left|\skp{\int_G\overline{\xi(\theta)}A_{\theta}\d\mu(\theta)\,h}{e_m}\right|^2\d\nu(\xi)\\
    &=\int_{\hat{G}}\sum_{m=1}^{\infty}\left|\skp{B_\xi h}{e_m}\right|^2\d\nu(\xi)=\int_{\hat G}\|B_\xi h\|^2\d\nu(\xi)\\
    &=\int_{\hat{G}}\skp{|B_\xi|^2h}{h}\d\nu(\xi)=\skp{\int_{\hat G}|B_\xi|^2\d\nu(\xi)h}{h}.
\end{align*}
The last equality, again, follows from (\ref{BochnerInner}), since $\skp{|B_\xi|^2h}{h}\in L^1(\hat{G},\nu)$ for all $h\in H$. Thus,
$$\int_{G}|A_{\theta}|^2\d\mu(\theta)=\int_{\hat G}|B_\xi|^2\d\nu(\xi).$$
\end{proof}

\begin{theorem}\label{TheoremPQ}
Let $G$ be a locally compact group with Haar measure $\mu$, let $\hat G$ be its Pontryagin dual with the dual measure $\nu$. For all $1\le p\le 2$, and $A_\theta\in L^1(G;\Sc_p)\cap L^2(G;\Sc_p)$ there holds
$B_{\xi}\in L^q(\hat{G},\Sc_p)$ and
\begin{equation}\label{Inequalityp<2}
\int_{\hat G}\left\|\int_GA_\theta\ovl{\xi(\theta)}\d\mu(\theta)\right\|_p^q\d\nu(\xi)\le
    \left(\int_G\|A_\theta\|_p^p\d\mu(\theta)\right)^{q/p},
\end{equation}
where $q$ is conjugate to $p$, i.e.\ $q=p/(p-1)$, $p\neq1$ and $q=+\infty$ if $p=1$.
\end{theorem}

\begin{remark}
    Of course, (\ref{Inequalityp<2}) doesn't make sense for $q=+\infty$. However, it should be considered as the power "to the $q$" is moved from the right hand side to the left and then the limit as $q\to+\infty$ is taken, i.e.
    $$\sup_{\xi\in\hat G}\left\|\int_GA_\theta\ovl{\xi(\theta)}\d\mu(\theta)\right\|_1\le\int_G\|A_\theta\|_1\d\mu(\theta).$$
    The same apply for corollaries stated in Theorems \ref{TheoremPQRn}, \ref{TheoremPQ2} and \ref{TheoremPQweighted}.
\end{remark}

\begin{proof} It is well known that for any $A\in\Sc_p$ we have $\|A\|\le\|A\|_p$. Hence $A_\theta\in L^1(G;\Sc_p)\cap L^2(G;\Sc_p)$ implies $A_\theta\in L^1(G;B(H))\cap L^2(G;B(H))$, and consequently, (\ref{AbstractParseval}) holds. Therefore, by (\ref{BochnerLinear}) with $T=\tr$, we obtain
\begin{align*}
    \int_{\hat G}\|B_{\xi}\|_2^2\d\nu(\xi)&=\int_{\hat G}\tr(|B_{\xi}|^2)\d\nu(\xi)=\tr\int_{\hat G}|B_{\xi}|^2\d\nu(\xi)=\tr\int_{G}|A_{\theta}|^2\d\mu(\theta)\\
&=\int_G\tr(|A_{\theta}|^2)\d\mu(\theta)=\int_G\|A_{\theta}\|_2^2\d\mu(\theta).
\end{align*}

It means that the linear mapping $A_\theta\mapsto B_\xi$, defined on the dense set $L^1(G;\Sc_2)\cap L^2(G;\Sc_2)$ in $L^2(G;\Sc_2)$ is an isometry. Hence, it can be extended to an isometry from the whole $L^2(G;\Sc_2)$ to $L^2(\hat G;\Sc_2)$. Denote this isometry by $T$.

On the other hand, if $A_\theta\in L^1(G;\Sc_1)\cap L^2(G;\Sc_1)$ (which is dense in $L^1(G;\Sc_1)$), then
$$\|B_\xi\|_1=\left\|\int_GA_\theta\overline{\xi(\theta)}\d\mu(\theta)\right\|_1\le\int_G\|A_\theta\|_1\d\mu(\theta)=\|A_\theta\|_{L^1(G;\Sc_1)},$$
and taking a supremum over all $\xi\in\hat G$
$$\|B_\xi\|_{L^\infty(\hat G;\Sc_1)}\le\|A_\theta\|_{L^1(G;\Sc_1)}.$$

This leads to a bounded, densely defined operator $L^1(G;\Sc_1)\supseteq L^1(G;\Sc_1)\cap L^2(G;\Sc_1)\ni A_\theta\mapsto B_\xi\in L^\infty(\hat G;\Sc_1)$, with norm at most $1$. It can be extended, by continuity, to a bounded operator on the whole $L^1(G;\Sc_1)$. Denote it again by $T$.

To summarize, we have a linear operator $T$, $T(A_\theta)=B_\xi$, defined on $L^1(G;\Sc_1)$ and $L^2(G;\Sc_2)$, which coincide on its intersection $L^1(G;\Sc_1)\cap L^2(\Sc_2)=L^1(G;\Sc_1)\cap L^2(\Sc_1)$ (since $\Sc_1\subseteq\Sc_2$) with an explicit form $B_\xi=\int_GA_\theta\overline{\xi(\theta)}\d\mu(\theta)$.

Now, we are ready to apply complex interpolation. Namely, by Corollary \ref{InterpolationLC}, we obtain
$$[L^1(G;\mathcal C_1),L^2(G,\mathcal C_2)]_{2-\frac{2}{p}}=L^p(G,\mathcal C_p),$$
$$[L^{\infty}(\hat{G};\mathcal C_1),L^2(\hat{G},\mathcal C_2)]_{2-\frac{2}{p}}=L^q(\hat{G},\mathcal C_p),$$
where $q$ is conjugated to $p$. Finally, by Theorem \ref{ExactInterpolation} we obtain $\|T\|\le 1$ considering $T$ as an operator from $L^p(G;\Sc_p)$ to $L^q(\hat G;\Sc_p)$. For $A_\theta\in L^1(G;\Sc_p)\cap L^2(G;\Sc_p)\subseteq L^p(G;\Sc_p)$ it becomes
$$\left(\int_{\hat G}\|B_\xi\|_p^q\d\nu(\xi)\right)^{\frac1q}\le\left(\int_G\|A_\theta\|_p^p\d\mu(\theta)\right)^{\frac1p},$$
i.e.~(\ref{Inequalityp<2}).
\end{proof}

\section{Applications}

Note, first, that in the case where the Hilbert space is one-dimensional, the function $A_\theta$ reduces to scalar valued function. Hence, our inequality (\ref{Inequalityp<2}) reduces to the Hausdorff-Young inequality \cite[(4.27)]{Foland}.

In the next two subsections, we apply our main results to some special cases of $G$, and to weighted Schatten classes.

\subsection{Applications by specifying the group $G$}

If $G$ is a compact group, then Theorem \ref{TheoremPQ} becomes \cite[Theorem 3.6]{Keckic2019}. Exactly which inequalities we can obtain for specific groups $G$ is described in section 4 of \cite{Keckic2019}. For instance, for $G=\Z_2$ we obtain (\ref{Lp<2}), and for $G=\Z_n$ (\ref{Cpn}). In \cite[Corollaries 4.5 and 4.7]{Keckic2019}, cases $G=\T$ and $G=\Z_2^n$ were also treated.

Next, we can take two important cases of non compact groups, $G=\R^n$ and $G=\Q_p^n$.

It is well known that $\R^n$ is a locally compact abelian group and that the Lebesque measure $\d m_n$ is its Haar measure. Its Pontryagin dual is also $\R^n$ with the dual measure $\frac1{(2\pi)^n}\d m_n$.

\begin{remark}
    Sometimes, in order to emphasize selfduality, the Haar measure is rescaled to $\frac1{(2\pi)^{n/2}}\d m_n$. Then it is equal to its dual measure.
\end{remark}

If we apply Theorem \ref{TheoremPQ}, we get

\begin{theorem}\label{TheoremPQRn}
    For all $1\le p\le 2$, and $A_\theta\in L^1(\R^n;\Sc_p)\cap L^2(\R^n;\Sc_p)$ there holds $B_{\xi}\in L^q(\R^n,\Sc_p)$ and
\begin{equation}\label{Inequalityp<2Rn}
\frac1{(2\pi)^n}\int_{\R^n}\left\|\int_{\R^n}A_\theta e^{-i\xi\theta}\d\theta\right\|_p^q\d\xi\le
    \left(\int_{\R^n}\|A_\theta\|_p^p\d\theta\right)^{q/p},
\end{equation}
where $q$ is conjugate to $p$, i.e.\ $q=p/(p-1)$, $p\neq1$ and $q=+\infty$ if $p=1$.
\end{theorem}

Our second application concerns $p$-adic groups.

Recall that given a fixed prime number $p$, a nonzero rational $x$ can be written as $x=r^k\frac ab$, with $p\not|ab$ and $k\in\Z$. Then we define its $p$-adic norm by
$$|x|_p=\begin{cases}p^{-k},&x\neq0\\0,&x=0.\end{cases}$$
The $p$-adic group $\Q_p$ is the completion of $\Q$ with respect to $|\cdot|_p$.

The nonzero elements of $\Q_p$ can be identified with formal infinite sums of the form
\begin{equation}\label{p-adic number}
x=\sum_{i=\gamma}^{\infty}x_ip^i,\qquad\gamma\in\mathbb{Z},\quad x_i\in\{0,1,...,p-1\},\quad x_\gamma\neq0,
\end{equation}
with $|x|_p=p^{-\gamma}$. The set of $p$-adic integers are
$$\Z_p=\{x\in\Q_p\mid |x|_p\le1\}.$$

It is well known that $\Q_p$ is a locally compact group, and hence it has its Haar measure, unique up to a scalar. It is usual to normalize Haar measure such that $\Z_p$ be of measure one. This Haar measure is denoted by $\d x$.

Though it is not easy to give an explicit form for Haar measure, it is not the case with the characters. Namely, if $x$ is given by (\ref{p-adic number}) then its fractional part $\{x\}_p$ is equal to $0$ if $\gamma\ge0$ (or $x\in\Z_p$) and
$$\{x\}_p=\sum_{i=\gamma}^{-1}x_ip^i,$$
otherwise. Then the standard character on $\Q_p$ is given by $\chi_p:\Q_p\to\T$, $\chi_p(x)=e^{2\pi i\{x\}_p}$. Obviously $\chi_p(x+y)=\chi_p(x)\chi_p(y)$.

It can be obtained that any character on $\Q_p$ is of the form $x\mapsto\chi_p(\xi x)$ for some $\xi\in\Q_p$. Thus $\hat\Q_p\cong\Q_p$.

The group $\Q_p^n$ can be treated in a similar way. It is also a locally compact group, with the Haar measure $\d^nx$ and $\hat\Q_p^n\cong\Q_p^n$ with the identification
$$\Q_p^n\ni\xi=(\xi_1,\dots,\xi_n)\leftrightarrow\chi_\xi\in\hat\Q_p^n,\qquad
\chi_\xi(x)=\chi(\xi\cdot x)=e^{2\pi i\{\xi_1x_1+\dots+\xi_n x_n\}_p}.$$

The Fourier transform of a function $f\in L^1(\Q_p^n)$ is given by
$$\hat f(\xi)=\int_{\Q_p^n}f(x)e^{2\pi i\xi\cdot x}\d^nx.$$

If we apply Theorem \ref{TheoremPQ} to the group $\Q_p^n$ we obtaine the following corollary. To avoid ambiguity, we write $q$ and $r$ as conjugated ceofficient instead of usual $p$ and $q$, since the letter $p$ is booked for the prime number.

\begin{theorem}\label{TheoremPQ2}
For all $1\le q\le 2$, and $A_\theta\in L^1(\mathbb{Q}_q^n;\Sc_q)\cap L^2(\mathbb{Q}_p^n;\Sc_q)$
$$B_\xi=\int_{\mathbb{Q}_p^n}e^{-2\pi i\{\sum_{i=1}^n x_i\e_i\}_p}A_\theta\d\mu(\theta)\in L^r(\mathbb{Q}_p^n,C_q),$$
and
\begin{equation}\label{Inequalityp<2Padic}
\int_{\mathbb{Q}_p^n}\left\|\int_{\mathbb{Q}_p^n}e^{-2\pi i\{\sum_{i=1}^n x_i\e_i\}_p}A_\theta\d\mu(\theta)\right\|_q^r\le
\left(\int_{\mathbb{Q}_p^n}\|A_\theta\|_q^q\d\mu(\theta)\right)^{r/q},
\end{equation}
where $r$ is conjugate to $q$, i.e.\ $r=q/(q-1)$.
\end{theorem}

\subsection{Weighted Schatten classes}

In \cite{Conde}, the weighted Schatten classes were investigated. As sets, they coincide with the usual Schatten classe, but they have a different norm.

Let $B(H)$ denote the algebra of bounded operators acting on a complex and separable Hilbert space $H$, $\Gl(H)$ the group of invertible elements of $B(H)$ and $\Gl(H)^{+}$ the set of all positive elements of $\Gl(H)$.

For the sake of simplicity, we denote with capital letters the elements of $\Sc_{p}$ and with lower case letters the elements of $\Gl(H)^{+}$.

On $\Sc_{p}$ we define the following norm associated with an invertible $a\in\Gl(H)^{+}$:
$$\|X\|_{p,a}:=\left\|a^{-1/2}Xa^{-1/2}\right\|_{p}.$$

The norm $\|\cdot\|_{p,a}$ is equivalent to $\|\cdot\|_p$, and hence the space $\Sc_{p,a}=(\Sc_p,\|\cdot\|_{p,a})$ is a Banach space. The exact interpolation space for different $a$ and $b$ is described in the following Lemma.

\begin{lemma}\label{L1}\cite[Corollary 3.2]{Conde}
Given invertible $a, b \in\Gl(H)^{+}$and $1\leqslant p<\infty$. Then
$$\left[\Sc_{p,a},\Sc_{p,b}\right]_{t}=\Sc_{p,\gamma_{a,b}(t)},$$
where $\gamma_{(a,b)}(t)=a^{1/2}(a^{-1/2}ba^{-1/2})^{t}a^{1/2}$.
\end{lemma}

For Bochner spaces, related to weighted Schatten ideals, by previous Lemma and Theorem \ref{T1} we immediately obtain:

\begin{lemma}\label{L2}
Let $p\geq1$ and $0<t<1$. Then $$[L^{p}(G;\Sc_{p,a}),L^{p}(S;\Sc_{p.b})]_{t}=L^p(G;\Sc_{p,\gamma_{a,b}(t)})$$ isometrically.
\end{lemma}

We formulate a version of Clarkson-McCarthy inequality for weighted Schatten classes.

\begin{theorem}\label{TheoremPQweighted}
Let $G$ be a locally compact group with Haar measure $\mu$, let $\hat G$ be its Pontryagin dual with the dual measure $\nu$. For all $1\le p\le 2$, $A_\theta\in L^1(G;\Sc_{p,a})\cap L^2(G;\Sc_{p,a})$, invertible $a$, $b\in\Gl(H)^{+}$, and $t \in[0,1]$, the operator
$$B_\xi=\int_GA_\theta\overline{\xi(\theta)}\d\mu(\theta)$$
belongs to $L^q(\hat{G};\Sc_{p,\gamma_{a,b}(t)})$ and
\begin{equation}\label{TPQweightedagamma}
\int_{\hat G}\left\|\int_GA_\theta\ovl{\xi(\theta)}\d\mu(\theta)\right\|_{p,\gamma_{a,b}(t)}^q\d\nu(\xi)
    \le\|a^{1/2}b^{-1}a^{1/2}\|^t\left(\int_G\|A_\theta\|_{p,a}^p\d\mu(\theta)\right)^{q/p},
\end{equation}
as well as for the same $p$, $a$, $b$, $t$ and $A_\theta\in L^1(G;\Sc_{p,\gamma_{a,b}(t)})\cap L^2(G;\Sc_{p,\gamma_{a,b}(t)})$ we have $B_\xi\in L^q(\hat G;\Sc_{p,a})$ and
\begin{equation}\label{TPQweightedgammaa}\int_{\hat G}\left\|\int_GA_\theta\ovl{\xi(\theta)}\d\mu(\theta)\right\|_{p,a}^q\d\nu(\xi)
\le\|b^{1/2}a^{-1}b^{1/2}\|^t\left(\int_G\|A_\theta\|_{p,\gamma_{a,b}(t)}^p\d\mu(\theta)\right)^{q/p},
\end{equation}
where $q$ is conjugate to $p$, i.e.\ $q=p/(p-1)$, $p\neq1$ and $q=+\infty$ if $p=1$.
\end{theorem}

\begin{proof}
Put $a^{-1/2}A_\theta a^{-1/2}$ instead of $A_\theta$ in (\ref{Inequalityp<2}). Applying (\ref{BochnerLinear}) to
$$A_\theta\mapsto a^{-1/2}A_\theta a^{-1/2},$$
we obtain
\begin{equation}\label{InequalityPQweighteda}
\int_{\hat G}\left\|a^{-1/2}\int_GA_\theta\ovl{\xi(\theta)}\d\mu(\theta)a^{-1/2}\right\|_p^q\d\nu(\xi)
    \le\left(\int_G\|a^{-1/2}A_\theta a^{-1/2}\|_p^p\d\mu(\theta)\right)^{q/p},
\end{equation}
i.e.
$$\int_{\hat G}\left\|\int_GA_\theta\ovl{\xi(\theta)}\d\mu(\theta)\right\|_{p,a}^q\d\nu(\xi)
    \le\left(\int_G\|A_\theta\|_{p,a}^p\d\mu(\theta)\right)^{q/p},$$
This means that the mapping $A_\theta\mapsto B_\xi=\int_GA_\theta\overline{\xi(\theta)}\d\mu(\theta)$ is a bounded linear operator, with norm at most one, regarded as an operator from $L^p(G;\Sc_{p,a})$ to $L^q(\hat G;\Sc_{p,a})$. Denote this operator by $T$. We have
\begin{equation}\label{Tweightedaa}
\|T\|_{L^p(G;\Sc_{p,a})\mapsto L^q(\hat G;\Sc_{p,a})}\le1.
\end{equation}

Now, put $b$ instead of $a$ in (\ref{InequalityPQweighteda}). We have
\begin{align*}\int_{\hat G}&\left\|b^{-1/2}\int_GA_\theta\ovl{\xi(\theta)}\d\mu(\theta)b^{-1/2}\right\|_p^q\d\nu(\xi)
    \le\left(\int_G\|b^{-1/2}A_\theta b^{-1/2}\|_p^p\d\mu(\theta)\right)^{q/p}\\
    &\le\left(\|b^{-1/2}a^{1/2}\|^p\int_G\|a^{-1/2}A_\theta a^{-1/2}\|_p^p\d\mu(\theta)\|b^{-1/2}a^{1/2}\|^p\right)^{q/p}\\
    &=\|a^{1/2}b^{-1}a^{1/2}\|^q\left(\int_G\|A_\theta\|_{p,a}^p\d\mu(\theta)\right)^{q/p}.
\end{align*}

The latter means that the operator $T$ is a bounded linear operator, regarded as an operator from $L^p(G;\Sc_{p,a})$ to $L^q(\hat G;\Sc_{p,b})$ with norm at most $\|a^{1/2}b^{-1}a^{1/2}\|$, i.e.
\begin{equation}\label{Tweightedab}
\|T\|_{L^p(G;\Sc_{p,a})\mapsto L^q(\hat G;\Sc_{p,b})}\le\|a^{1/2}b^{-1}a^{1/2}\|.
\end{equation}

From (\ref{Tweightedaa}) and (\ref{Tweightedab}), using Theorem \ref{ExactInterpolation} and Lemma \ref{L2} we obtain
$$\|T\|_{L^p(G;\Sc_{p,a})\mapsto L^q(\hat G;\Sc_{p,\gamma_{a,b}(t)})}\le\|a^{1/2}b^{-1}a^{1/2}\|^t,$$
i.e.\ (\ref{TPQweightedagamma}).

The inequality (\ref{TPQweightedgammaa}) can be obtained by switching the roles of $a$ and $b$ in (\ref{Tweightedab}) and Theorem \ref{ExactInterpolation} and Lemma \ref{L2}.
\end{proof}

\bibliographystyle{amsplain}
\bibliography{Clarkson}
\end{document}